\documentclass[12pt]{article}

\usepackage{pgfplots}
\usepackage{times}
\usepackage{amsmath,amsfonts, amstext,amssymb,amsbsy,amsopn,amsthm}
\usepackage{dsfont}
\usepackage{esint}
\usepackage{graphicx}   
\usepackage{hyperref}
\usepackage[all]{xy}
\usepackage{color}
\usepackage{tikz}
\usepackage{lineno,hyperref}

\usepackage{pgfplots}

\newtheorem{definition}{Definition}[section]
\newtheorem{theorem}[definition]{Theorem}

\theoremstyle{remark}
\newtheorem{remark}[definition]{Remark}
\numberwithin{equation}{section}

\newcommand{\abs}[1]{\lvert#1\rvert}

\newcommand{\norm}[1]{\lVert#1\rVert}

\newcommand{\R}{\mathbb{R}}
\newcommand{\cL}{\mathcal{L}}
\newcommand{\lap}{\mbox{$\triangle$}}
\newcommand{\ma}[1]{(\partial_ t - \triangle)^{#1}}

\newcommand{\fr}{\displaystyle\frac}
\newcommand{\jf}{\displaystyle\int}

\newcommand{\mb}{\mbox}

\newcommand{\be}{\begin{equation}}
\newcommand{\ee}{\end{equation}}
\newcommand{\bee}{\begin{equation*}}
\newcommand{\eee}{\end{equation*}}

\begin{document}

\title{  A convergence result for the master operator}

\author{Wenxiong Chen, Yahong Guo,
 Congming Li, Yugao Ouyang}
\date{}
\maketitle
\begin{abstract}
In this paper, we establish a  convergence result for the fully fractional heat operator $\ma{s}$, also known as the master operator,  stated as follows:
\[\mbox{If\ }u_i\to u\ \mbox{in}\ C^{2,1}_{x,t,loc}(\R^n\times\R),\ \mbox{then}\ \ma{s} u_i\to \ma{s}u-b\ \mbox{a.e. in}\ \R^n\times\R,\]
for some nonnegative constant  $b$. This result addresses a fundamental question in the blow-up and rescaling analysis, which are essential for establishing a priori estimates for
solutions of  master equations.

Additionally, we present examples demonstrating that in certain cases, the constant $b$ can indeed be positive. This highlights a key distinction between nonlocal and local operators: for a local heat operator, such as $\partial_t - \lap$, it is well-known that $b \equiv 0$. \end{abstract}
\smallskip

\medskip

\medskip

\section{Introduction}

It is well-known that obtaining a priori estimate for solutions of certain equations, such as
\be \label{a1}  \mathcal{L} v(x) = f(v(x)), \;\; x \in \Omega \subset \R^n, \ee
often relies on blow-up and rescaling analysis, a standard yet powerful technique. During this process, the rescaled sequence $\{u_i\}$, along with its derivatives up to the same order as the operator $\mathcal{L}$, can be made to converge on any compact set to a function $u$ and its corresponding derivatives.

A crucial question then arises:

\centerline{\em Does  $\mathcal{L} u_i(x)$  converges to  $\mathcal{L} u(x)$ pointwise in $\R^n$? }

If the answer is affirmative, one can derive a limiting equation and potentially reach a contradiction with known nonexistence result, leading to the desired a priori estimate for solutions of original equation \eqref{a1}.

For the classical Laplace operator $\Delta$, the answer is positive.  Specifically,
\[\mbox{if\ }u_i\to u\ \mbox{in}\ C^2_{loc}(\R^n)\ \mbox{as}\ i\to\infty,\ \mbox{then}\ \Delta u_i(x) \to \Delta u(x)\ \mbox{ pointwise in } \R^n.\]

However, this property fails for the fractional Laplacian $(-\Delta)^s$ (see \cite{DXJY}), defined as
$$ (-\Delta)^s u(x)=C_{n, s} \,\text{P.V.}\int_{\mathbb{R}^n}\frac{u(x)-u(y)}{|x-y|^{n+2s}}dy,$$
where $0<s<1$ and P.V. denotes the Cauchy principal value.

This breakdown is primarily due to the nonlocal nature of the fractional Laplacian, where the interaction between points across arbitrarily large distances prevents the pointwise  convergence of fractional Laplacian of the sequence $(-\lap)^s u_i(x)$  to that of the limit function $(-\lap)^s u(x)$.

Observe  that the singular integral defining the fractional Laplacian consists of two essential components: one over the ball of radius $R$,
denoted as  $B_R$, and the other over its complement $B_R^c$. The condition
\be \label{a2} u_i\to u\ \mbox{in}\ C^2_{loc}(\R^n) \ee
ensures only the convergence of the first part,
$$ \int_{B_R}\frac{u_i(x)-u_i(y)}{|x-y|^{n+2s}}dy \to \int_{B_R}\frac{u(x)-u(y)}{|x-y|^{n+2s}}dy.$$

However, this condition provides no information regarding the convergence of the second part,
$$ I_i \equiv \int_{B_R^c}\frac{u_i(x)-u_i(y)}{|x-y|^{n+2s}}dy.$$

When the original domain $\Omega$ is unbounded, even  under condition \eqref{a2}, the sequence $\{u_i(x)\}$ may become unbounded near infinity, while its point-wise limit $u(x)$ remains  bounded or smooth.

In \cite{DXJY}, the authors examined  this phenomenon and, under appropriate regularity assumptions, prove that

If $\{ u_i \}$ is a sequence of nonnegative functions converging to a function $u \in \mathcal{L}^{2s} (\R^n)$, and $\{(-\lap)^s u_i\}$ converges pointwisely in $\R^n$, then there exists a constant $b\geq 0$ such that
\be \label{a3} \lim_{i \to \infty} (-\lap)^s u_i(x) = (-\lap)^s u(x) - b, \;\; \forall \, x \in \R^n. \ee

They also provide  an example demonstrating  that the constant $b$ can be nonzero in certain cases. Here
$$ \mathcal{L}^{2s} (\R^n) \equiv \{ u \in L^1_{loc}(\R^n) \mid \int_{\R^n} \frac{|u(x)|}{1 + |x|^{n+2s}} d x < \infty \}$$
is a space of slowly increasing functions, and $u \in C^{1,1}_{loc}(\R^n) \cap \mathcal{L}^{2s}$ is a condition to ensure the convergence of the singular integral defining the fractional Laplacian.

\medskip

In this paper, we will investigate this kind of convergence phenomenon for the fully fractional heat operator $\ma{s}$ defined as

    \begin{equation}\label{nonlocaloper}
\ma{s}u(x,t):=c_{n,s}\,\text{P.V.}\int _{-\infty}^t\int _{\R ^n}\frac{u(x,t)-u(y,\tau)}{(t-\tau)^{\frac{n}{2}+1+s}}e^{-\frac{\abs{x-y}^2}{4(t-\tau)}}dyd\tau,
    \end{equation}
where $s\in(0,1),$ P.V. stands for the  Cauchy principal value, and $c_{n,s}$ is
the normalization constant. Hereafter, the P.V. notation is omitted for brevity, with all integrals understood in the Cauchy principal value sense.

The operator $\ma{s}$, also known as the  mater operator, was initially introduced by M. Riesz in \cite{Riesz}.
It is a space-time nonlocal pseudo-differential operator of order $2s$ in the spatial variable $x$ and of order $s$ in the time variable $t$, as the value of $(\partial_t-\Delta)^s u$ at any point $(x,t)$ depends on the values of $u$ over the entire spatial domain $\mathbb{R}^n$ and for all  past times up to $t$.

This operator plays a significant role in various physical and biological applications, including anomalous diffusion \cite{KBS}, chaotic dynamics \cite{Z}, and biological invasions \cite{BRR}. One notable application of the master equation is to describe continuous-time random walks, where \( u \) represents the distribution of particles undergoing random jumps along with random time delays (cf. \cite{MK}).

It is worth noting that the fractional powers of heat operator $(\partial_t-\Delta)^s$
reduces to the classical heat operator $\partial_t-\Delta$ as $s\rightarrow 1$ (cf. \cite{FNW}) in the sense that, for each $(x,t)$,
$$ (\partial_t-\Delta)^s u(x,t) \to (\partial_t-\Delta) u(x,t), \;\; \mbox{ as } s \to 1.$$

Moreover,
when the space-time nonlocal operator $(\partial_t-\Delta)^s$ is applied to a function $u$ depending only on the spatial variable $x$, it simplifies to
 \begin{equation*}
   (\partial_t-\Delta)^s u(x)=(-\Delta)^s u(x),
 \end{equation*}
where $(-\Delta)^s$ is the widely recognized fractional Laplacian. This operator is of greater interest due to its diverse applications across various scientific disciplines. In recent decades, considerable attention has been dedicated to study the fractional elliptic equations and a series of fruitful results have been obtained.  Interested readers may refer to \cite{CLL, CLL1, CLM, CZhu, DLL,  DQ, DY,GP, H, HZ, LZhuo, LZ, ZZZ} and references therein.

While in the special case when $u$ depends only on the time variable $t$, we have:
 \begin{equation*}
   (\partial_t-\Delta)^s u(t)=\partial_t^s u(t),
 \end{equation*}
where $\partial_t^s$ denotes the Marchaud fractional derivative of order $s$, defined as \begin{equation}\label{1.00}
\partial^s_t u(t)=C_s \jf_{-\infty}^t\fr{u(t)-u(\tau)}{(t-\tau)^{1+s}}d\tau.
\end{equation}
This derivative arises in various physical phenomena, such as
particle systems with sticking and trapping effects, magneto-thermoelastic heat conduction, plasma turbulence and more (cf. \cite{ACV,ACV1, DCL1, DCL2, EE}).

To ensure that the singular integral in \eqref{nonlocaloper} converges,  we assume
$$u(x,t)\in C^{2s+\epsilon,s+\epsilon}_{x,\, t,\, {\rm loc}}(\mathbb{R}^n\times\mathbb{R}) \cap \mathcal{L}^{2s,s}(\mathbb{R}^n\times\mathbb{R})$$
for some $\varepsilon >0$, where $C^{2s+\epsilon,s+\epsilon}_{x,\, t,\, {\rm loc}}(\mathbb{R}^n\times\mathbb{R})$ represents the local parabolic H\"{o}lder space, as defined in \cite{CGL}, consisting of functions which are locally $C^{2s+\epsilon}$ in the spatial variable $x$ and $C^{s+\epsilon}$ in the temporal variable $t$.
While the slowly increasing function space $\mathcal{L}^{2s,s}$ is defined by
$$\mathcal{L}^{2s,s}(\mathbb{R}^n\times\mathbb{R})=\left\{u(x,t) \in L^1_{\rm loc} (\mathbb{R}^n\times\mathbb{R}) \mid \int_{-\infty}^{+\infty} \int_{\mathbb{R}^n} \frac{|u(x,t)|}{1+|x|^{n+2+2s}+|t|^{\frac{n}{2}+1+s}}\operatorname{d}\!x\operatorname{d}\!t<\infty\right\}. $$

If
$$u(x,t)\in C^{2s+\epsilon,s+\epsilon}_{x,\, t,\, {\rm loc}}(\mathbb{R}^n\times\mathbb{R}) \cap \mathcal{L}^{2s,s}(\mathbb{R}^n\times\mathbb{R})$$
and satisfies the equation
\be \label{a5} \ma{s} u(x,t) = f(x,  u(x,t)), \;\; (x,t) \in \R^n \times \R \ee
pointwise, then we call it the classical solution (and simply the solution) of \eqref{a5}.

One of the major motivations for the present paper comes from establishing the a priori estimate for positive solutions to \eqref{a5}, or in the case $1/2 <s <1$, for more general master equations of the form
 \begin{equation} \label{a6}
 (\partial_t - \Delta)^s u(x,t) = b(x)|\nabla_x u(x,t)|^q + f(x, u(x,t)), \;\; (x,t) \in \mathbb{R}^n \times \mathbb{R},
 \end{equation}
where $\nabla_x u$ is the gradient of $u$ with respect to $x$.

Consider \eqref{a5} as an example. To show that there exists a constant $C$, independent of $u$,  such that
\be \label{a7} u(x,t) \leq C, \;\; \forall \, (x,t) \in \R^n \times \R, \ee
an effective approach is via blow-up and rescaling argument briefly as follows.

Suppose \eqref{a7} is violated, then there exist a sequence of solutions $u_k$ of \eqref{a5} and a sequence of points  $(x^k, t_k)$ such that
$$u_k(x_k, t_k) \to \infty, \mbox{ as } k \to \infty. $$

After a proper rescaling,
$$
   v_k(x,t)=\frac{1}{M_k}u_k(\lambda_kx+\bar{x}_k, \lambda_k^2 t + \bar{t}_k),
 $$
 we can derive that, for any $R>0$, $v_k$ is uniformly bounded in
 $$Q_R:=\{(x,t)\in \R ^n\times\R:\abs{x}\leq R, \abs{t}\leq R^2\},$$
 the standard parabolic cylinder of size $R$ centered at $(0,0)$.
Furthermore, it can be deduced from \eqref{a5} that $v_k$ solves
\be \label{eq4-8}
(\partial_t -\Delta)^s v_k(x,t) = F_k(x, v_k(x,t)),\,\, (x,t)\in Q_{R}
\ee
with $F_k$ being uniformly H\"{o}lder continuous. Applying our new interior regularity estimates (see \cite{CGL}), we can conclude that
the sequence $\{v_k(x,t)\}$ converges in $C^{2s+\epsilon, s+\epsilon}_{x,t}$ to a function $v(x,t)$ on any compact subset of $\R^n \times \R$. Nonetheless, during the rescaling process, as is often the case, no information is available about $\{v_k(x,t)\}$ outside $Q_R$, and the sequence may become unbounded near infinity.

In order to arrive at a limiting equation, it is critical and fundamental to determine whether $(\partial_t - \lap)^s v_k(x,t)$ converges and to identify the limit. We address this question in the following theorem, which constitutes the main result of this paper.  For convenience, we denote the sequence $\{v_k\}$ by $\{u_i\}$ in what follows.

\begin{theorem}\label{main thm}
    Let $n\ge 1$, $s\in (0,1)$ and $\varepsilon >0$ sufficiently small. Suppose that
    $$\{u_i\}\subset \mathcal{L}^{2s,s} (\R ^n\times\R)\cap C^{2s+\varepsilon,s+\frac{\varepsilon}{2}} _{x,t}(Q_{R_i})$$ is a sequence of nonnegative functions, with $\{R_i\}$ be a sequence of positive numbers converging to $+\infty$. If $\{u_i\}$ converges in $C^{2s+\varepsilon,s+\frac{\varepsilon}{2}} _{x,t,loc}(\R ^n\times\R)$ to a function $u\in \mathcal{L}^{2s,s} (\R ^n\times\R)$,  $\{\ma{s}u_i\}$ converges pointwise in $\R ^n\times\R$,  and  there exists a universal constant $M>0$ such that for all $(x,t)\in\R^n\times\R,$
\begin{equation}\label{bdd}
 \liminf\limits_{i\to\infty}\ma{s}[u_i-u](x,t)\geq -M,
   \end{equation}then there exists a nonnegative constant $b$ such that
\begin{equation}
      \lim _{i\to\infty}\ma {s}u_i(x,t) = \ma{s}u(x,t)-b, \quad \forall\,(x,t)\in  \R ^n\times\R.
\end{equation}
Moreover,
\[
b = c_{n,s}\lim _{R\to\infty}\lim_{i\to\infty}\int _{\left(\R^n\times(-\infty,0)\right)\setminus Q_R}\dfrac{u_i(y,\tau)}{|\tau|^{\frac{n}{2}+1+s}}e^{\frac{\abs{y}^2}{4|\tau|}}dyd\tau,\,\,(\text{the limit exists and is finite}).
\]
\end{theorem}
\begin{remark}
    In the process of deriving a priori estimates for solutions to the master equation on unbounded domains using the blow-up and rescaling method, the condition \eqref{bdd}  is satisfied automatically (see \cite{CGL}).
\end{remark}
\begin{remark}
    In Theorem \ref{main thm}, we assume $u_i$ and $u$  belong to the  space $\mathcal{L}^{2s,s} (\R ^n\times\R)$ of tempered distributions. In fact, our result remains valid under the weaker assumption that $u_i\in\cL (\R ^n\times\R)$ with the limiting function $u\in \mathcal{L}^{2s,s} (\R ^n\times\R)$.
    Here, $$ \mathcal{L}(\mathbb{R}^n\times\mathbb{R}):=\left\{u(x,t) \in L^1_{\rm loc} (\mathbb{R}^n\times\mathbb{R}) \mid \int_{-\infty}^t \int_{\mathbb{R}^n} \frac{|u(x,\tau)|e^{-\frac{|x|^2}{4(t-\tau)}}}{1+(t-\tau)^{\frac{n}{2}+1+s}}\operatorname{d}\!x\operatorname{d}\!\tau<\infty,\,\, \forall \,t\in\mathbb{R}\right\}$$
is  typically used  to define the classical solutions to master equations. Although the space $\cL (\R ^n\times\R)$ may seem broader than $\mathcal{L}^{2s,s} (\R ^n\times\R)$, the latter is a natural generalization of $\cL^{2s}(\R^n)$ associated with the fractional Laplacian.

More specifically,  if $u$ is a function of $x$ only, i.e., $u(x,t)=u(x)$,  a straightforward calculation shows:
\[\int_{-\infty}^{+\infty}\int_{\mathbb{R}^n} \frac{|u(x)|}{1+|x|^{n+2+2s}+|t|^{\frac{n}{2}+1+s}}\operatorname{d}\!x\operatorname{d}\!t\sim\int_{\mathbb{R}^n} \frac{|u(x)|}{1+|x|^{n+2s}}\operatorname{d}\!x.\]
This relation highlights that the space $\mathcal{L}^{2s,s}(\mathbb{R}^n\times\mathbb{R})$ is consistent with the space
$\mathcal{L}^{2s}(\R^n)$.

From this perspective, Theorem \ref{main thm} generalize the known results on the fractional Laplacian in \cite{DXJY}. However, this is by no means a straight forward generalization. Our situation here is significantly more intricate due to the interplay between space and time variables, as we will illustrate shortly.
\end{remark}

\begin{remark}
As we mentioned earlier, if $u$ depends only on $t$, then the master operator reduces to the Marchaud fractional time derivative
$$(\partial_t - \lap)^s u(t) = \partial_t^s u(t).$$
Thus, our result applies to the Marchaud derivatives and still remains novel.
\end{remark}
\medskip

\leftline{\bf The difficulties and our novel approaches}
\medskip

As will be shown in more details in the next section, there are three terms in the difference
$$ \ma {s}u(x,t)-\ma {s}u_i(x,t) = I_i(x,t,R)+E_i(x,t,R)+F_i(x,t,R).$$
Through standard estimates, it is straight forward to show that
$$ I_i(x,t,R)\to 0,\,\,\text{as}\,\,i\to\infty$$ and
$$ \lim_{R\to\infty}\sup _{(x,t)\in Q_{R/2}}\abs{\lim_{i \to \infty} E(x,t,R)}=0.$$

Then what left is the dominant remainder term
$$F_i(x,t,R) :=  c_{n,s}\int _{\left(\R^n\times(-\infty,t)\right)\setminus Q_R} \frac{u_i(y,\tau)}{(t-\tau)^{\frac{n}{2}+1+s}}e^{-\frac{\abs{x-y}^2}{4(t-\tau)}}dyd\tau.$$  Our goal is to prove  that for $(x,t)\in Q_{{R}/{3}}$,
\[ F(x,t,R):=\lim\limits_{i\to\infty}F_i(x,t,R) \] converges to a  constant as $R\to \infty$.
\smallskip

    In the case of the fractional Laplacian, the corresponding term is
 $$F_i(x,R) := C_{n,s} \int_{B_R^c(0)} \frac{u_i(y)}{|x-y|^{n+2s}} d y.$$
 The exterior region is simple, so is the kernel function
 \[K(x):=\frac{C_{n,s}}{|x|^{n+2s}}.\]

 For each fixed $x\in B_R$ and for  all $y\in B_R^c$, the kernel  satisfies the following inequality
\[ \left(\frac{R}{R+|x|}\right)^{n+2s}\leq\frac{K(x-y)}{K(0-y)}=\frac{|y|^{n+2s}}{|x-y|^{n+2s}}\leq \left(\frac{R}{R-|x|}\right)^{n+2s}, \]
and hence
$$ \frac{K(x-y)}{K(0-y)} \to 1, \;\; \mbox{ uniformly as } R \to \infty. $$
It follows that
$$ \frac{F_i(x,R)}{F_i(0,R)} \to 1, \;\; \mbox{ as } R \to \infty.$$
From here it is easy to derive that, there is some constant $b$, such that
$$  (-\lap)^s u(x) - (-\lap)^s u_i(x) \to b, \;\; \mbox{ as } i \to \infty. $$

For our master equation, two major challenges arise:

\begin{enumerate}

\item {\em Complexity of the exterior region:} The set $\left(\R^n\times(-\infty,t)\right)\setminus Q_R$ has a complicated structure.

\item {\em Strong spatiotemporal correlation in the kernel function:} The kernel function   \[
    M(x,t): =\frac{c_{n,s}}{t^{\frac{n}{2}+1+s}}e^{-\frac{\abs{x}^2}{4t}},\quad\forall\,t>0, x\in\R ^n,
    \]
    exhibits a strong coupling between space and time variables.
 \end{enumerate}

To circumvent these difficulties, we introduced several novel techniques. Specifically,   we analyze the ratio of the kernel function $M(x-y,t-\tau)$ at pairs of well-chosen values of $(x,t)$ while   $(y,\tau)$ is varying in  the exterior  region $\left(\R^n\times(-\infty,t)\right)\setminus Q_R$. This approach allows us to establish sharp estimates for the ratio of  kernel functions across different  parts of the exterior  region.

To illustrate this approach, denote
$$b(x,t) = \lim_{i \to \infty} [\ma {s}u(x,t)-\ma {s}u_i(x,t)].$$
In the first step, we demonstrate that $b(\cdot, \cdot)$  is independent of $x$, then in the second step, we show that it is also independent of $t$.
\bigskip

In the proof of $b(x,t)=b(0,t)$, through careful deliberations, we chose proper hyper surfaces to divide the exterior region into three parts $A_R,B_R,$ and $C_R$. Figure \ref{fig:partition x} below illustrates this partition in the two-dimensional case, where $(x,t)$ is positioned at the origin, and $Q^{-}_R:=Q_R\cap \R^n_{-}$.

\begin{figure}[htbp]
    \centering
    \begin{tikzpicture}[scale=1.2]

        \fill[gray!20] (-1,0) -- (1,0) -- (1,-3)-- (-1,-3) -- cycle;
        \fill[blue!20] (-1,-3) -- (1,-3) -- (1,-4.5) -- (-1,-4.5) -- cycle;
        \fill[purple!20] (-1,-2) -- (-2,-4) -- (-2,-4.5) -- (-1,-4.5) -- cycle;
        \fill[purple!20] (1,-2) -- (2,-4) -- (2,-4.5) -- (1,-4.5) -- cycle;
        \fill[green!20] (-1,0) -- (-1,-2) -- (-2,-4) -- (-2,0) -- cycle;
        \fill[green!20] (1,0) -- (1,-2) -- (2,-4) -- (2,0) -- cycle;

        \node at (1.5,-1.6) {$A_R$};
        \node at (-1.5,-1.6) {$A_R$};
        \node at (-0.5, -3.6) {$C_R$};
        \node at (0.5, -1.6) {$Q_R^-$};
        \node at (-1.5, -3.6) {$B_R$};
        \node at (1.5, -3.6) {$B_R$};
        \node at (-2.1,-4.2) {$y=R^{-\frac{1}{3}}\tau$};
        \node at (2.1,-4.2) {$y=-R^{-\frac{1}{3}}\tau$};

        \draw[thick] (0, 0) -- (2, -4);
        \draw[thick] (0, 0) -- (-2, -4);
        \draw[thick] (-1, -3) -- (-1, -4.5);
        \draw[thick] (1,-3) -- (1, -4.5);
        \draw[dashed] (-1,0) -- (-1,-3) -- (1,-3) -- (1,0);

        \draw[->, thick] (-2.2, 0) -- (2.2, 0) node[right] {$y$};
        \draw[->, thick] (0, -5) -- (0, 0.5) node[above] {$\tau$};

    \end{tikzpicture}
    \caption{2-dim. case, partition of outer domain when $(x,t)=(0,0)$.}
    \label{fig:partition x}
\end{figure}
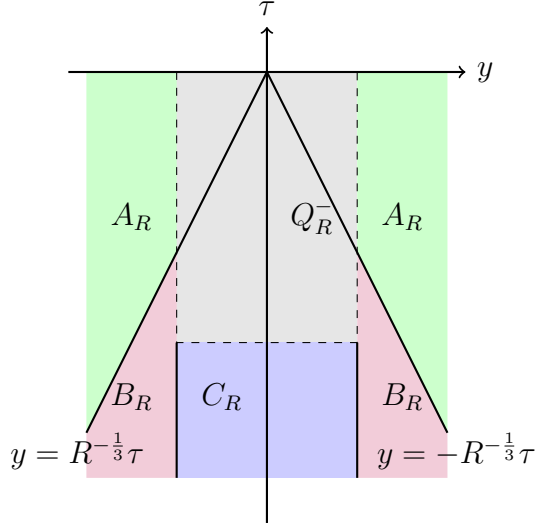
\bigskip

We establish the following estimates: \begin{equation*}
    \frac{M(x-y,t-\tau)}{\sum_{j=i}^nM(x\pm R^{\frac{2}{3}}{e_j}-y,t-\tau)}\leq e^{-c{R^{\frac{1}{3}}}}\to 0\,\,\text{as }R\to \infty, \quad \forall\, (y,\tau)\in A_R,
\end{equation*}
and
\begin{equation*}
    \frac{M(x-y,t-\tau)}{M(0-y,t-\tau)}\to 1\,\,\text{as}\ R\to\infty\ \mbox{ uniformly\ in}\   (y,\tau)\in B_R\cup C_R.
\end{equation*}\bigskip

Then  in the proof of  $b(0,t)=b(0,0)$, we  divide the exterior region into four parts $C_R,D_R,E_R,$ and $F_R$ (see Figure \ref{fig:partition t} for the two dimensional case).  We derive that
\begin{equation*}
    \frac{M(0-y,0-\tau)}{M(0-y,t-\tau)}\to 1\,\,\text{as}\ R\to\infty\ \mbox{ uniformly\ in}\  (y,\tau)\in C_R\cup D_R,
\end{equation*}
and
\[
\frac{M(0-y,t-\tau)}{M(0-y,t+R^{\frac{3}{2}}-\tau)}\leq{R^{-\frac{1}{2}}}\to 0\,\,\text{as }R\to \infty, \ \forall (y,\tau)\in E_R\cup F_R.
\]

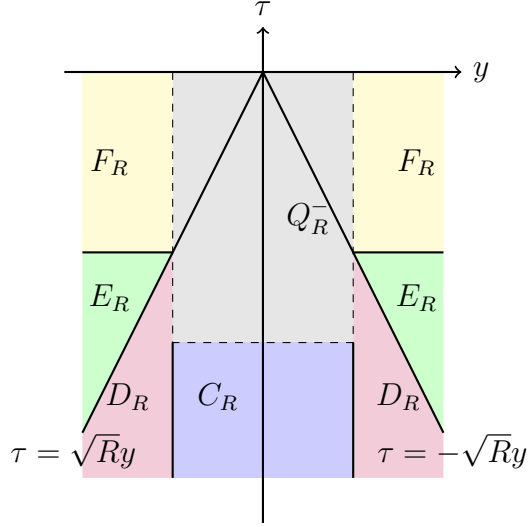
\begin{figure}[htbp]
    \centering
    \begin{tikzpicture}[scale=1.2]

        \fill[gray!20] (-1,0) -- (1,0) -- (1,-3)-- (-1,-3) -- cycle;
        \fill[blue!20] (-1,-3) -- (1,-3) -- (1,-4.5) -- (-1,-4.5) -- cycle;
        \fill[purple!20] (-1,-2) -- (-2,-4) -- (-2,-4.5) -- (-1,-4.5) -- cycle;
        \fill[purple!20] (1,-2) -- (2,-4) -- (2,-4.5) -- (1,-4.5) -- cycle;
        \fill[green!20] (-1,-2) -- (-2,-2) -- (-2,-4) -- cycle;
        \fill[green!20] (1,-2) -- (2,-2) -- (2,-4) -- cycle;
        \fill[yellow!20] (-2,0) -- (-1,0) -- (-1,-2) -- (-2,-2) -- cycle;
        \fill[yellow!20] (2,0) -- (1,0) -- (1,-2) -- (2,-2) -- cycle;

        \node at (1.7,-2.5) {$E_R$};
        \node at (1.7,-1) {$F_R$};
        \node at (-1.7,-1) {$F_R$};
        \node at (-1.7,-2.5) {$E_R$};
        \node at (-0.5, -3.6) {$C_R$};
        \node at (0.5, -1.6) {$Q_R^-$};
        \node at (-1.5, -3.6) {$D_R$};
        \node at (1.5, -3.6) {$D_R$};
        \node at (-2.1,-4.2) {$\tau=\sqrt{R}y$};
        \node at (2.1,-4.2) {$\tau=-\sqrt{R}y$};

        \draw[thick] (0, 0) -- (2, -4);
        \draw[thick] (0, 0) -- (-2, -4);
        \draw[thick] (-1, -3) -- (-1, -4.5);
        \draw[thick] (1,-3) -- (1, -4.5);
        \draw[thick] (-2,-2)--(-1,-2);
        \draw[thick] (1,-2)--(2,-2);
        \draw[dashed] (-1,0) -- (-1,-3) -- (1,-3) -- (1,0);

        \draw[->, thick] (-2.2, 0) -- (2.2, 0) node[right] {$y$};
        \draw[->, thick] (0, -5) -- (0, 0.5) node[above] {$\tau$};

    \end{tikzpicture}
    \caption{2-dim. case, partition of outer domain when $t=0$.}
    \label{fig:partition t}
\end{figure}

Moreover, in domains $A_R, E_R, F_R$, the  estimate  established in \eqref{Fibdd} (see the next section), namely,\begin{equation*}
         0\le \sup _{(x,t)\in Q_{R/2}}F(x,t,R)\le M+1 \text{ whenever }R>>1.
     \end{equation*}
plays a crucial role in our proof, where $M$ is the universal constant given in \eqref{bdd}.
\bigskip

We also provide some examples to show that the constant $b$ in Theorem \ref{main thm}  can indeed be positive.
\begin{theorem}\label{main thm2}
    There exists a sequence of nonnegative smooth functions \{$w_i$\} with compact support in $\R^n\times\R$ such that
    \[\lim\limits_{i\to\infty}w_i=0,\ \mbox{in\ the \ sense\ of}\ C_{x,t,loc}^{2,1}(\R^n\times\R),\]
    but\[\lim\limits_{i\to\infty}\ma{s}w_i(x,t)=-1,\ \forall (x,t)\in\R^n\times\R.\]
\end{theorem}

\bigskip

This paper is organized as follows.
In Section 2, we prove the convergence for the fully fractional heat operator $\ma{s}$ as stated in Theorem \ref{main thm}.
In Section 3, we present some examples to demonstrate that the constant $b$ obtained in Theorem \ref{main thm} can be positive and thus establish Theorem \ref{main thm2}.
\section{Convergence}

    Throughout the proof, unless explicitly stated, $C, C_i,c_i,c $  to denote universal constants whose values may vary from line to line. We write $A\lesssim B$
if $A\le CB$ holds for some constant $C>0$,
possibly varying in different contexts.

\begin{proof}[Proof of Theorem\ref{main thm}]
    Fix $(x,t)\in \R^n\times \R$. Select $R\gg 3\max\{\sqrt{\abs{t}}, \abs{x}\}$ and consider sufficiently large $i$. We commence by decomposing the difference:
    \begin{align}
        &\ma {s}u(x,t)-\ma {s}u_i(x,t) \nonumber\\=&\int _{-\infty}^t\int _{\R^n}\left((u-u_i)(x,t)-(u-u_i)(y,\tau) \right)M(x-y,t-\tau)dyd\tau \nonumber\\
        =&\int _{-R^2}^t\int _{B_R}\left((u-u_i)(x,t)-(u-u_i)(y,\tau) \right)M(x-y,t-\tau)dyd\tau \nonumber\\
        +& \int _{\left(\R^n\times(-\infty,t)\right)\setminus Q_R} \left((u-u_i)(x,t)-u(y,\tau) \right)M(x-y,t-\tau)dyd\tau \nonumber\\
        +& \int _{\left(\R^n\times(-\infty,t)\right)\setminus Q_R} u_i(y,\tau)M(x-y,t-\tau)dyd\tau \nonumber\\
        :=& I_i(x,t,R)+E_i(x,t,R)+F_i(x,t,R)\label{1}.
    \end{align}

    Regarding the term $I_i$, for $2s+\varepsilon<1$,  for any $(y,\tau)\in B_R\times (-R^2,t)$, we have
    \[
    \left|(u-u_i)(x,t)-(u-u_i)(y,\tau)\right|\le [u-u_i]_{C^{2s+\varepsilon,s+\frac{\varepsilon}{2}}_{x,t}(Q_{R})}\left(|x-y|^{2s+\varepsilon} + (t-\tau)^{s+\frac{\varepsilon}{2}}\right).
    \]
    By changing of variables, we deduce
    \begin{align*}
        |I_i(x,t,R)|&\le [u-u_i]_{C^{2s+\varepsilon,s+\frac{\varepsilon}{2}}_{x,t}(Q_{R})}\int _{0}^{2R^2}\int _{\R ^n}\frac{|z|^{2s+\varepsilon}+a^{s+\frac{\varepsilon}{2}}}{a^{\frac{n}{2}+1+s}}e^{-\frac{\abs{z}^2}{4a}}dzda\\
        &\le C(R)[u-u_i]_{C^{2s+\varepsilon,s+\frac{\varepsilon}{2}}_{x,t}(Q_{R})}.
    \end{align*}

    For $2s+\varepsilon>1$, we denote $v_i = u-u_i$. Then
    \begin{align*}
        I_i(x,t,R) & = \int _{-R^2}^t\int _{B_R}\frac{v_i(x,t)-v_i(x,\tau)}{(t-\tau)^{\frac{n}{2}+1+s}}e^{-\frac{|x-y|^2}{4(t-\tau)}}dyd\tau\\
        &+\int _{-R^2}^t\int _{B_R\backslash B_1(x)}\frac{v_i(x,\tau)-v_i(y,\tau)}{(t-\tau)^{\frac{n}{2}+1+s}}e^{-\frac{|x-y|^2}{4(t-\tau)}}dyd\tau\\
        &+ \int _{-R^2}^t\int _{ B_1(x)}\frac{v_i(x,\tau)-v_i(y,\tau)}{(t-\tau)^{\frac{n}{2}+1+s}}e^{-\frac{|x-y|^2}{4(t-\tau)}}dyd\tau\\
        :&= I^1_i+I_i^2+I_i^3.
    \end{align*}
    For $I_i^1$, a direct calculation shows
    \begin{align*}
        |I_i^1|&\le \int _{-R^2}^t\frac{|v_i(x,t)-v_i(x,\tau)|}{(t-\tau)^{1+s}}\left(\int _{\R ^n}\frac{e^{-\frac{|x-y|^2}{4(t-\tau)}}}{(t-\tau)^{\frac{n}{2}}}dy\right)d\tau\\
        &\le C\int _{-R^2}^t\frac{|v_i(x,t)-v_i(x,\tau)|}{(t-\tau)^{1+s}}d\tau\\
        &\le C\int _{-R^2}^t\frac{\norm {v_i}_{C^{2s+\varepsilon,s+\frac{\varepsilon}{2}}_{x,t}(Q_{R})}(t-\tau)^{s+\frac{\varepsilon}{2}}}{(t-\tau)^{1+s}}\\
        &\le C(R)\norm {v_i}_{C^{2s+\varepsilon,s+\frac{\varepsilon}{2}}_{x,t}(Q_{R})}.
    \end{align*}
    For $I_i^2$, we have 
    \[
    |v_i(x,\tau)-v_i(y,\tau)|\le \|v_i\|_{C^{2s+\varepsilon,s+\frac{\varepsilon}{2}}_{x,t}(Q_{R})}|x-y|\le \|v_i\|_{C^{2s+\varepsilon,s+\frac{\varepsilon}{2}}_{x,t}(Q_{R})}|x-y|^{2s+\varepsilon}.
    \]
    Similar argument to the case $2s+\varepsilon<1$ implies
    \[
    |I_i^2|\le C(R)\|v_i\|_{C^{2s+\varepsilon,s+\frac{\varepsilon}{2}}_{x,t}(Q_{R})}.
    \]
    For $I_i^3$, we rewrite
    \begin{align*}
        \int _{-R^2}^t&\int _{B_1(x)}\frac{v_i(x,\tau)-v_i(y,\tau)}{(t-\tau)^{\frac{n}{2}+1+s}}e^{-\frac{|x-y|^2}{4(t-\tau)}}dyd\tau\\&= \frac{1}{2}\int _{-R^2}^t\int _{B_1}\frac{(v_i(x,\tau)-v_i(x-y,\tau))-(v_i(x+y,\tau)-v_i(x,\tau))}{(t-\tau)^{\frac{n}{2}+1+s}}e^{-\frac{|y|^2}{4(t-\tau)}}dyd\tau.
    \end{align*}
    For any $(y,\tau)\in B_{1}\times (-R^2,t)$, there exist $\xi _1$ and $\xi _2$ lying on the line segment connecting $x$ to $x-y$ and the line segment connecting $x$ to $x+y$, respectively, such that
    \[
    v_i(x,\tau)-v_i(x-y,\tau) = \nabla_x(\xi_1,\tau)\cdot y,\quad v_i(x+y,\tau)-v_i(x,\tau) = \nabla_x(\xi_2,\tau)\cdot y.
    \]
    Hence,
    \begin{align*}
        |I_3^i| &\le\int _{-R^2}^t\int _{B_{1}}\frac{|(\nabla_x(\xi_1,\tau)-\nabla_x(\xi_2,\tau))\cdot y|}{(t-\tau)^{\frac{n}{2}+1+s}}e^{-\frac{|y|^2}{4(t-\tau)}}dyd\tau
        \\&\le \int _{-R^2}^t\int _{B_{1}}\frac{\norm{v_i}_{C^{2s+\varepsilon,s+\frac{\varepsilon}{2}}_{x,t}(Q_{R})}|\xi _1-\xi _2|^{2s+\varepsilon-1}|y|}{(t-\tau)^{\frac{n}{2}+1+s}}e^{-\frac{|y|^2}{4(t-\tau)}}dyd\tau\\
        &\le C\norm{v_i}_{C^{2s+\varepsilon,s+\frac{\varepsilon}{2}}_{x,t}(Q_{R})}\int _{-R^2}^t\int _{B_{1}}\frac{|y|^{2s+\varepsilon}}{(t-\tau)^{\frac{n}{2}+1+s}}e^{-\frac{|y|^2}{4(t-\tau)}}dyd\tau\\
        & \le C(R)\norm{v_i}_{C^{2s+\varepsilon,s+\frac{\varepsilon}{2}}_{x,t}(Q_{R})}.
    \end{align*}
     
    To sum up, we have proved that for any $s\in (0,1)$,
    \[
    |I_i(x,t,R)|\le C(R)\norm{u-u_i}_{C^{2s+\varepsilon,s+\frac{\varepsilon}{2}}_{x,t}(Q_{R})}.
    \]
    Therefore, from the assumption that $u_i\to u$ in $C^{2s+\varepsilon,s+\frac{\varepsilon}{2}} _{x,t}(Q_{R})$, we obtain
    \begin{equation}\label{2}
         I_i(x,t,R)\to 0,\,\,\text{as}\,\,i\to\infty.
    \end{equation}

   Turning to the analysis of $E_i$,  we rewrite
    \begin{align*}
        E_i(x,t,R) & = \int _{\left(\R^n\times(-\infty,t)\right)\setminus Q_R} \left((u-u_i)(x,t)-u(y,\tau) \right)M(x-y,t-\tau)dyd\tau \\
        &= (u-u_i)(x,t)\left(\int _{-\infty}^{-R^2}\int _{\R^n}M(x-y,t-\tau)dyd\tau+\int _{B_R^c}\int_{-R^2}^tM(x-y,t-\tau)dyd\tau\right)\\&-\int  _{\left(\R^n\times(-\infty,t)\right)\setminus Q_R}u(y,\tau)M(x-y,t-\tau)dyd\tau.
    \end{align*}
   First, it can be checked that
    \[\int _{-\infty}^{-R^2}\int _{\R^n}M(x-y,t-\tau)dyd\tau+\int _{B_R^c}\int_{-R^2}^tM(x-y,t-\tau)dyd\tau\leq C/R^{2s},\]
then by $u_i(x,t)\to u(x,t)$,
    we derive
    \begin{equation}\label{3}
       E(x,t,R):= \lim _{i\to\infty}E_i(x,t,R) = -\int  _{\left(\R^n\times(-\infty,t)\right)\setminus Q_R}u(y,\tau)M(x-y,t-\tau)dyd\tau.
    \end{equation}
    Next, we claim that
     \begin{equation}\label{uni E}
         \lim _{R\to\infty}\sup _{(x,t)\in Q_{R/2}}\abs{E(x,t,R)}=0.
     \end{equation}
   To see this, we decompose the error term as
     \begin{align*}
         |E(x,t,R)| =& \left(\int _{-\infty}^{-R^2}\int _{\R ^n} +\int _{-R^2}^t\int _{|y|>R}\right)|u(y ,\tau)|M(x-y,t-\tau)dyd\tau\\
         :=&J_1+J_2.
     \end{align*}
Since the kernel admits a pointwise decay estimate (See also \cite{ST}): there exists a constant $\Lambda = \Lambda(n,s)>0$, such that
    \begin{equation*}
        M(x,t)\le \frac{\Lambda}{|x|^{n+2+2s}+t^{\frac{n}{2}+1+s}}.
    \end{equation*}
  Thus,  for any $(x,t)\in Q_{\frac{R}{2}}$,  we obtain
     \begin{align}
         |J_1| &\le \int _{-\infty}^{-R^2}\int _{\R ^n}\frac{\Lambda|u(y,\tau)|}{|t-\tau|^{\frac{n}{2}+1+s}+|x-y|^{n+2+2s}}dyd\tau\nonumber\\
         & \le C\int _{-\infty}^{-R^2}\int _{\R ^n}\frac{\abs{u(y,\tau)}}{\abs{\tau}^{\frac{n}{2}+1+s}+|x-y|^{n+2+2s}}dyd\tau.\label{I1}
     \end{align}
     In the last inequality, we have used the fact that $|t-\tau|\ge |\tau|/2$, since $|\tau|\ge R^2\ge 2|t|$.

    For spatial estimates, consider two cases:
      \begin{itemize}
        \item For $|y|> 2|x|$, since $x\in B_{R/2}$, we derive $|x-y|\ge |y|/2$ and thus
        \[
        1+|\tau|^{\frac{n}{2}+1+s}+|y|^{n+2+2s}\lesssim |\tau|^{\frac{n}{2}+1+s}+|x-y|^{n+2+2s}.
        \]
        \item For $|y|\le 2|x|$, the condition $|\tau|\ge R^2\ge (2|x|)^2$ implies
        \[
        1+|\tau|^{\frac{n}{2}+1+s}+|y|^{n+2+2s}\lesssim |\tau|^{\frac{n}{2}+1+s}.
        \]
    \end{itemize}
Combining the above two inequalities with \eqref{I1}, we deduce
     \[
     \sup _{(x,t)\in Q_{R/2}}|I_1(x,t)|\le C\int _{-\infty}^{-R^2}\int _{\R ^n}\frac{|u(y,\tau)|}{1+|\tau|^{\frac{n}{2}+1+s}+|y|^{n+2+2s}}dyd\tau.
     \]
     Since $u\in\mathcal{L}^{2s,s} (\R ^n\times\R)$, by the dominated convergence theorem, we arrive at
     \begin{equation}\label{J1}
         \lim _{R\to\infty}\sup _{(x,t)\in Q_{R/2}}|J_1(x,t)| =0.
     \end{equation}

     Similarly, \begin{align}\label{J2}
         |J_2| \le C \int _{-R^2}^{R^2}\int _{|y|>R}\frac{|u(y,\tau)|}{|t-\tau|^{\frac{n}{2}+1+s}+|y|^{n+2+2s}}dyd\tau .
    \end{align}
     For temporal parameters $\tau\in (-R^2,R^2)$, we distinguish two scenarios:
    \begin{itemize}
        \item When $|\tau|>2|t|$, we obtain $|t-\tau|\ge |\tau|/2$ and then
        \[
        1+|\tau|^{\frac{n}{2}+1+s}+|y|^{n+2+2s}\lesssim |t-\tau|^{\frac{n}{2}+1+s}+|y|^{n+2+2s}.
        \]
        \item For $|\tau|\le 2\abs{t}$, since $|y|\geq R,$ we have
        \[
        1+|\tau|^{\frac{n}{2}+1+s}+|y|^{n+2+2s}\lesssim |y|^{n+2+2s}.
        \]
    \end{itemize}
   Consequently,  \eqref{J2} further transforms into
    \[
    |J_2|\lesssim \int _{-\infty}^{\infty}\int _{|y|>R}\frac{|u(y,\tau)|}{1+|\tau|^{\frac{n}{2}+1+s}+|y|^{n+2+2s}}dyd\tau .
    \]
    The dominated convergence theorem again demonstrates
    \begin{equation}\label{{I22}}
         \lim _{R\to\infty}\sup _{(x,t)\in Q_{R/2}}|J_2(x,t)| =0.
    \end{equation}
   Combing \eqref{J1} and  \eqref{{I22}}, we obtain  \eqref{uni E}.

     Now, since $\{\ma{s}u_i(x,t)\}$ converges, we infer from \eqref{1}, \eqref{2} and \eqref{3} that
     \[F(x,t,R):=\lim _{i\to\infty}F_i(x,t,R)\text{ exists and is finite}. \]
     According to the definition of $F_i$ and $u_i\ge 0$, we have $$F(x,t,R)\ge 0.$$
     Moreover,
     \begin{equation}
        F(x,t,R) = \ma{s}u(x,t)-\lim _{i\to\infty}\ma{s}u_i(x,t)-E(x,t,R).
     \label{RLT}
     \end{equation}

According to the assumption \eqref{bdd} and \eqref{uni E},  there exists a universal constant $N$ independent of $x$ and $t$, such that,
     \begin{equation}\label{Fibdd}
         0\le \sup _{(x,t)\in Q_{R/2}}F(x,t,R)\le M+1 \text{ whenever }R>N.
     \end{equation}
     Observing that $F_i(x,t,R)$ is nonnegative and non-increasing in $R$, so is $F(x,t,R)$. Hence, the limit
     \begin{equation}\label{def-b}
         b(x,t):=\lim _{R\to\infty}F(x,t,R).
     \end{equation}
  exists and is nonnegative and  finite.

     Therefore, combining \eqref{uni E}, \eqref{RLT}, and \eqref{def-b},  it suffices to show that
     \begin{equation}\label{bx=b}
         b(x,t)\equiv b,
     \end{equation}
with a nonnegative constant $b.$

    To this end, we will proceed the proof in two steps.

    Step 1: We show that $b(x,t)$ is independent of $x$, i.e.,
    \begin{equation}
        b(x,t) = b(0,t),\quad\forall \,(x,t)\in Q_R.
    \end{equation}

The core idea is to estimate, for each fixed $t$,  the ratio of the kernel $ M(x - y, t - \tau) $ at pairs of well chosen spatial points, with $ (y, \tau) $ varying in different parts of the exterior region $ \left( R^n \times (-\infty, t) \right) \setminus Q_R $. To achieve this, we  divide the region into two parts: $\{(y,\tau):\abs{y}\ge R\}$ and
$\{(y,\tau):\abs{y}\le R, \abs{\tau}\ge R^2\}$.

(1) $\abs{y}\ge R$. Take $\delta = {R}^{-\frac{1}{3}}$. Note that
\[
M(x-y,t-\tau) = \frac{c_{n,s}}{(t-\tau)^{\frac{n}{2}+1+s}}e^{-\frac{\abs{x-y}^2}{4(t-\tau)}},\quad\forall\,\tau<t.
\]

Case 1: For $\abs{x-y}\ge \delta\abs{t-\tau} $, we show that
\begin{equation}\label{c1}
    \frac{M(x-y,t-\tau)}{M(x\pm \frac{e_j}{\delta ^2}-y,t-\tau)}\lesssim \exp\{-\frac{c}{\delta
    }\}\to 0\,\,\text{as }R\to 0, \quad \forall\, y\in I_{\pm}^j, \tau<t.
\end{equation}
For the notation $\{I_+^j,I_-^j\}_{j=1}^n$, we mean that,
\[
I_+^j:=\{y\in\R^n:y_j-x_j\ge 0,\,\,\text{and } y_j-x_j = \max _{1\le k\le n}|y_k-x_k| \}.
\]
\[
I_-^j:=\{y\in\R^n:y_j-x_j\le 0,\,\,\text{and } x_j-y_j = \max _{1\le k\le n}|y_k-x_k| \}.
\]
Here we divide the spatial coordinate system centered at $x$ into $2n$ equal parts, as shown in  Figure \ref{fig:partition1} for $n=2.$

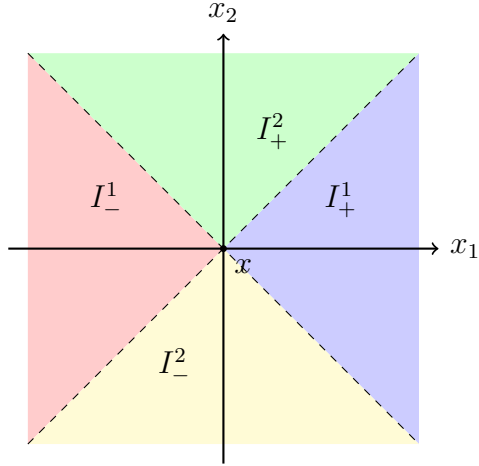
\begin{figure}[htbp]
    \centering
    \begin{tikzpicture}[scale=1.3]

        \fill[blue!20] (0,0) -- (2,2) -- (2,-2) -- cycle; 
        \fill[red!20] (0,0) -- (-2,2) -- (-2,-2) -- cycle; 
        \fill[green!20] (0,0) -- (2,2) -- (-2,2) -- cycle; 
        \fill[yellow!20] (0,0) -- (-2,-2) -- (2,-2) -- cycle; 

        \node at (1.2, 0.5) {$I_+^1$};
        \node at (-1.2, 0.5) {$I_-^1$};
        \node at (0.5, 1.2) {$I_+^2$};
        \node at (-0.5, -1.2) {$I_-^2$};

        \draw[dashed] (-2, -2) -- (2, 2);
        \draw[dashed] (-2, 2) -- (2, -2);

        \draw[->, thick] (-2.2, 0) -- (2.2, 0) node[right] {$x_1$}; 
        \draw[->, thick] (0, -2.2) -- (0, 2.2) node[above] {$x_2$}; 

        \fill[black] (0,0) circle (1pt) node[below right] {$x$};
    \end{tikzpicture}
    \caption{Partition of \(\mathbb{R}^2\) into regions \(I_+^j\) and \(I_-^j\).}
    \label{fig:partition1}
\end{figure}

Without loss of generality, we may assume $y\in I_+^j$, then
\begin{align*}
    \frac{M(x-y,t-\tau)}{M(x +\frac{e_j}{\delta ^2}-y,t-\tau)} & = \exp\left\{-\frac{\abs{x-y}^2}{4(t-\tau)}+\frac{\left|\frac{e_j}{\delta ^2} +x-y\right|^2}{4(t-\tau)}\right\}\\
    & = \exp\left\{-\frac{2(y_j-x_j)-\frac{1}{\delta ^2}}{4(t-\tau)\delta ^2}\right\}\\&\lesssim \exp\left\{-\frac{c_1\abs{y-x}-\frac{1}{\delta ^2}}{4(t-\tau)\delta ^2}\right\}\\&\le \exp\left\{-\frac{c_2\abs{y-x}}{4(t-\tau)\delta ^2}\right\}\\
    &\leq e^{-\frac{c}{\delta
    }}\to 0\,\, \text{as $\delta\to 0$ (or $R\to\infty$)}.
\end{align*}
Here we have  used the fact that
\[
\abs{y-x}\ge \abs{y}-\abs{x}\ge R-\frac{R}{3}= \frac{2}{3\delta ^3}\ge \frac{c}{\delta ^2}.
\]
This validate  \eqref{c1}.

Case 2: For $\abs{y-x}\leq \delta \abs{t-\tau}$, we show that
\begin{equation}\label{c2}
    \frac{M(x-y,t-\tau)}{M(0-y,t-\tau)}\to 1\,\,\text{as $R\to\infty$ uniformly in $y$ and for $\tau<t$}.
\end{equation}
Indeed,
\begin{align*}
    \frac{M(x-y,t-\tau)}{M(0-y,t-\tau)} &=\exp\left\{ -\frac{\abs{x-y}^2}{4(t-\tau)} +\frac{\abs{y}^2}{4(t-\tau)} \right\}\\
    & = \exp\left\{\frac{2x\cdot y - \abs{x}^2}{4(t-\tau)}\right\}\\&\le \exp\left\{\frac{2\abs{x}\abs{y} + \abs{x}^2}{4(t-\tau)}\right\}\\
    &\leq \exp\left\{\frac{2\delta\abs{x}(t-\tau)+3\abs{x}^2}{4(t-\tau)}\right\}\\
    & = \exp\left\{\frac{\delta \abs{x}}{2}\right\}\cdot \exp\left\{\frac{3\abs{x}^2}{4(t-\tau)}\right\}\\&\le \exp\left\{\delta (\frac{\abs{x}}{2}+\frac{3|x|^2}{4})\right\}\to 1,\,\,\text{as $\delta\to 0$ (or $R\to\infty$)}.
\end{align*}
Here, we have used the fact that
\[
\delta(t-\tau)\ge \abs{y-x}\ge \abs{y}-\abs{x}\ge R-\frac{R}{3}\ge 1.
\]
Similarly,
\[
\frac{M(0-y,t-\tau)}{M(x-y,t-\tau)}\le \exp\left\{\delta \left(\frac{\abs{x}}{2}+\frac{3|x|^2}{4} \right)\right\}\to 1\,\,\text{as $R\to \infty$}.
\]

(2) For $\abs{y}\le R,\abs{\tau}\ge R^2$, we show that
\begin{equation}\label{c3}
    \frac{M(x-y,t-\tau)}{M(0-y,t-\tau)}\to 1\,\,\text{as $R\to\infty$ uniformly in $y$ and for $\tau<t$}.
\end{equation}
Indeed, similar to Case 2 in (1), we have
\begin{align*}
    \frac{M(x-y,t-\tau)}{M(0-y,t-\tau)}&\le \exp\left\{\frac{2\abs{x}\abs{y} + \abs{x}^2}{4(t-\tau)}\right\}\\&\le \exp\left\{\frac{\abs{x}^2+2R\abs{x}}{4(R^2+t)}\right\}\\&\le \exp\left\{\frac{c\abs{x}}{R}\right\}\to 1\,\,\text{as $R\to\infty$}.
\end{align*}
Similarly,
\[
\frac{M(x-y,t-\tau)}{M(0-y,t-\tau)}\ge \exp\left\{-\frac{c\abs{x}}{R} \right\}\to 1\,\,\text{as $R\to \infty$}.
\]
Hence, we verified \eqref{c2} and \eqref{c3}.

Now divide the outer domain into three parts (as presented in Figure \ref{fig:partition x}):
\begin{align*}
    &A_R:=\{(y,\tau)\in B_R^c\times (-\infty,t):\abs{y-x}\ge \delta (t-\tau)\},\\
    &B_R:=\{(y,\tau)\in B_R^c\times (-\infty,t):\abs{y-x}\le \delta (t-\tau)\},\\
    &C_R:=B_R\times (-\infty,-R^2).
\end{align*}
Then
\begin{align}
    F_i(x,t,R) =& \int _{\left(\R^n\times(-\infty,t)\right)\setminus Q_R} u_i(y,\tau)M(x-y,t-\tau)dyd\tau\nonumber\\
    =& \int _{A_R} u_i(y,\tau)M(x-y,t-\tau)dyd\tau +\int _{B_R} u_i(y,\tau)M(x-y,t-\tau)dyd\tau\nonumber\\+&\int _{C_R} u_i(y,\tau)M(x-y,t-\tau)dyd\tau\nonumber\\
    :=& F_i^1(x,t,R)+F_i^2(x,t,R)+F_i^3(x,t,R).\label{F}
\end{align}

For $F_i^1(x,t,R)$, it follows from  \eqref{c1} that
\begin{align}
    F_i^1(x,t,R) & =\int _{A_R} u_i(y,\tau)M(x-y,t-\tau)dyd\tau\nonumber\\
    & = \sum _{j=1}^n\int _{A_R\cap (I_+^j\cup I_-^j)}u_i(y,\tau)\frac{M(x-y,t-\tau)}{M(x\pm \frac{e_j}{\delta ^2}-y,t-\tau)}M\left(x\pm \frac{e_j}{\delta ^2}-y,t-\tau\right)dyd\tau\nonumber\\
    & \le \exp\{-\frac{c}{\delta}\}\sum _{j=1}^n\int _{\left(\R^n\times(-\infty,t)\right)\setminus Q_R}u_i(y,\tau)M\left(x\pm \frac{e_j}{\delta ^2}-y,t-\tau\right)dyd\tau\nonumber\\
    &= \exp\{-\frac{c}{\delta}\}\sum _{j=1}^n F_i\left(x\pm \frac{e_j}{\delta ^2},t,R\right),\label{F11}
\end{align}
taking $i\to\infty$ in \eqref{F11}, we obtain
\begin{align}\label{F1}
     \limsup _{i\to\infty}F_i^1(x,t,R)&\le \exp\{-\frac{c}{\delta}\}\sum _{j=1}^nF\left(x\pm \frac{e_j}{\delta ^2},t,R\right).
\end{align}
Since $|x|\leq \frac{R}{3}$ and $\delta = {R}^{-\frac{1}{3}}$, we have $$|x\pm \frac{e_j}{\delta ^2}|\leq \frac{R}{3}+{R}^{\frac{2}{3}}\leq \frac{R}{2},$$    for sufficiently large $R$,  then by \eqref{Fibdd}, we deduce
\[
F\left(x\pm \frac{e_j}{\delta ^2},t,R\right)\le M+1.
\]
Together with \eqref{F11} and \eqref{F1}, we arrive at
\begin{equation}\label{F-1}
   \lim _{R\to\infty}\limsup _{i\to\infty} F_i^1(x,t,R)=0.
\end{equation}
Similarly,
\begin{equation}\label{F0}
   \lim _{R\to\infty}\limsup _{i\to\infty} F_i^1(0,t,R)=0.
\end{equation}

Now we estimate
\begin{align*}
    F_i^2 (0,t,R) &= \int _{B_R} u_i(y,\tau)M(x-y,t-\tau)dyd\tau\\
    &=\int _{B_R} u_i(y,\tau)M(0-y,t-\tau)\frac{M(x-y,t-\tau)}{M(0-y,t-\tau)} dyd\tau.
\end{align*}
 By \eqref{c2}, we derive
\[\exp\{-\delta(c_1\abs{x}^2+c_2\abs{x})\}F_i^2(0,t,R)\le F_i^2(x,t,R)\le \exp\{\delta(c_1\abs{x}^2+c_2\abs{x})\}F_i^2(0,t,R).
\]
Similarly, by \eqref{c3}, we have
\[
\exp\{-\frac{c\abs{x}}{R}\}F_i^3(0,t,R)\le F_i^3(x,t,R)\le \exp\{\frac{c\abs{x}}{R}\}F_i^3(0,t,R).
\]
Consequently,
\[ F_i^2(x,t,R)+F_i^3(x,t,R)\le \exp\{\delta c(\abs{x}^2+\abs{x})\}(F_i^2(0,t,R)+F_i^3(0,t,R)),
\]
and
\[ F_i^2(x,t,R)+F_i^3(x,t,R)\ge \exp\{-\delta c(\abs{x}^2+\abs{x})\}(F_i^2(0,t,R)+F_i^3(0,t,R)).\]
It follows that
\begin{equation}\label{F-23}
    \exp\{-\delta c(\abs{x}^2+\abs{x})\}\le \frac{F_i^2(x,t,R)+F_i^3(x,t,R)}{F_i^2(0,t,R)+F_i^3(0,t,R)}\le \exp\{\delta c(\abs{x}^2+\abs{x})\}.
\end{equation}
For any fixed $(x,t)\in Q_{\frac{R}{3}}$, we have
$$\delta c(\abs{x}^2+\abs{x}) =o(1)\ \mbox{for}\,R\ \mbox{large}, $$
Then from \eqref{F-23}, we deduce
\begin{equation}\label{F23}
     \lim_{R\to\infty}\limsup _{i\to\infty}(F^2_i(x,t,R)+F^3_i(x,t,R)) = \lim_{R\to\infty}\limsup _{i\to\infty}(F^2_i(0,t,R)+F^3_i(0,t,R)).
\end{equation}
On one hand, noticing that $F^1_i\ge0$, we have
\begin{align*}
    b(x,t) &=\lim_{R\to\infty}\limsup _{i\to\infty}(F_i^1+F_i^2+F_i^3)(x,t)\\
    &\ge\lim_{R\to\infty}\limsup _{i\to\infty}(F_i^2+F_i^3)(x,t).
\end{align*}
On the other hand, by \eqref{F-1},
\begin{align*}
    b(x,t) &=\lim_{R\to\infty}\limsup _{i\to\infty}(F_i^1+F_i^2+F_i^3)(x,t)\\
    &\le\lim_{R\to\infty}\limsup _{i\to\infty}F_i^1(x,t)+\lim_{R\to\infty}\limsup _{i\to\infty}(F_i^2+F_i^3)(x,t)\\
    &= \lim_{R\to\infty}\limsup _{i\to\infty}(F_i^2+F_i^3)(x,t).
\end{align*}
Therefore 
\[
b(x,t)=\lim_{R\to\infty}\limsup _{i\to\infty}(F_i^2+F_i^3)(x,t).
\]

Combining \eqref{F}, \eqref{F-1}, \eqref{F0} and \eqref{F23}, we obtain
\begin{align*}
    b(x,t)=\lim_{R\to\infty}\limsup _{i\to\infty}F_i(x,t,R) & = \lim_{R\to\infty}\limsup _{i\to\infty} (F_i^2(x,t,R)+F_i^3(x,t,R))\\
    &=\lim_{R\to\infty}\limsup _{i\to\infty} (F_i^2(0,t,R)+F_i^3(0,t,R))\\
    &=\lim_{R\to\infty}\lim _{i\to\infty} F_i(0,t,R) =:b(0,t).
\end{align*}

Step 2: In the previous step, we have prove that $b(x,t)=b(0,t)$, now we show that $b(0,t)$ is also independent of $t$, i.e.,
\[
b(0,t)=b(0,0):=b\ge 0.
\]

Similar to the first step, the core strategy involves analyzing the ratio  of the kernel $M(0 - y, t - \tau)$ at pairs of  well chosen  temporal variables, while the integration variables $(y, \tau)$ range over the exterior region $\left( \mathbb{R}^n \times (-\infty, t) \right) \setminus Q_R$.

To achieve this, we  divide the exterior region into four parts: $C_R$, $D_R$, $E_R$, and $F_R$ (as shown in Figure \ref{fig:partition t}),
\begin{align*}
    &C_R:=B_R\times (-\infty,-R^2),
    \\
    &D_R:=\{(y,\tau)\in B_R^c\times (-\infty,t):\abs{t-\tau}^2\ge R\abs{y}^2 \},\\
    &E_R:=\{(y,\tau)\in B_R^c\times (-\infty,-R^{\frac{3}{2}}):\abs{t-\tau}^2\le R\abs{y}^2 \},\\
    &F_R := B_R^c\times (-R^{\frac{3}{2}},t).
\end{align*}
Correspondingly, we have
\begin{align}
    F_i(0,t,R)=&\left(\int _{C_R}+\int _{D_R}+\int _{E_R}+\int _{F_R} \right)u_i(y,\tau)M(0-y,t-\tau)dyd\tau\nonumber\\
    :=&F_i^4(0,t,R)+F_i^5(0,t,R)+F_i^6(0,t,R)+F_i^7(0,t,R).\label{sFt}
\end{align}
Next, we show that for any $|t|\leq \frac{R^2}{9}$, it holds
\begin{align}\label{Ft}\nonumber
   &\lim _{R\to\infty}\limsup _{i\to\infty}(F_i^4(0,t,R)+F_i^5(0,t,R)) =\lim _{R\to\infty}\limsup _{i\to\infty}(F_i^4(0,0,R)+F_i^5(0,0,R)),  \\
   &\limsup_{i\to\infty} F_i^6(0,t,R)\to 0, \limsup_{i\to\infty}F_i^7(0,t,R)\to 0\,\,\text{as }R\to\infty,
\end{align}

We firstly prove: for sufficientle large $R$,
\begin{equation}\label{F4}
    (1+o(1))e^{-\frac{c}{R^2}}F^4_i(0,0,R)\le F^4_i(0,t,R)\le (1+o(1))e^{\frac{c}{R^2}}F^4_i(0,0,R).
\end{equation}

For $(y,\tau)\in C_R$, we have $\abs{y}<R$, and $\abs{\tau}\ge R^2$, and then $\abs{\tau}^2\ge R^2\abs{y}^2$. Consequently,
\begin{align*}
    \frac{M(-y,0-\tau)}{M(-y,t-\tau)} & = \left(\frac{t-\tau}{-\tau} \right)^{\frac{n}{2}+1+s} \exp\left\{-\frac{\abs{y}^2}{4}\frac{t}{(-\tau)(t-\tau)}\right\}\\
    &\le \left(\frac{t-\tau}{-\tau} \right)^{\frac{n}{2}+1+s}\exp\left\{\frac{c\abs{y}^2}{\tau^2}\right\}\\&\le (1+o(1))\exp\{\frac{c}{R^2}\}.
\end{align*}
Therefore,
\begin{align*}
    F^4_i(0,t,R) &= \int _{C_R}u_i(y,\tau)M(-y,0-\tau)\frac{M(t-y,-\tau)}{M(0-y,0-\tau)}dyd\tau\\&\ge (1+o(1))\exp\{-\frac{c}{R^2}\}F^4_i(0,0,R)
\end{align*}
Similarly, we have
\[F^4_i(0,t,R)\le (1+o(1))\exp\{\frac{c}{R^2}\}F^4_i(0,0,R).\]
Thus, \eqref{F4} is valid.

For $(y,\tau)\in D_R$, we show that
\begin{equation}\label{F5}
    (1+o(1))e^{-\frac{c}{R}}F^5_i(0,0,R)\le F^5_i(0,t,R)\le (1+o(1))e^{\frac{c}{R}}F^5_i(0,0,R).
\end{equation}
In this case, we have $\abs{\tau}^2\ge R^{3}$.
Similar to the situation in $C_R$, we deduce
\begin{align*}
    \frac{M(-y,0-\tau)}{M(-y,t-\tau)} & = \left(\frac{t-\tau}{-\tau} \right)^{\frac{n}{2}+1+s} \exp\left\{-\frac{\abs{y}^2}{4}\frac{t}{(-\tau)(t-\tau)}\right\}\\
    &\le \left(\frac{t-\tau}{-\tau} \right)^{\frac{n}{2}+1+s}\exp\left\{\frac{c\abs{y}^2}{\tau^2}\right\}\\&\le (1+o(1))\exp\{\frac{c}{R}\}.
\end{align*}
Then the same argument implies \eqref{F5}.

Combining \eqref{F4} and \eqref{F5}, we derive the first claim in \eqref{Ft}:
\begin{equation*}
    \lim _{R\to\infty}\limsup _{i\to\infty}(F_i^4(0,t,R)+F_i^5(0,t,R)) =\lim _{R\to\infty}\limsup _{i\to\infty}(F_i^4(0,0,R)+F_i^5(0,0,R)).
\end{equation*}

From now on, without loss of generality, we assume $t=0$.

For $(y,\tau)\in E_R$, consider $t_0>0$ to be chosen, then
\[
\frac{M(-y,0-\tau)}{M(-y,t_0-\tau)}=\left(\frac{t_0-\tau}{-\tau} \right)^{\frac{n}{2}+1+s}\exp\left\{-\frac{\abs{y}^2}{4}\frac{t_0}{(-\tau)(t_0-\tau)}\right\}
\]
In this case, $\abs{\tau}^2\leq R\abs{y}^2$, $\abs{\tau}\ge R^{\frac{3}{2}}$. We choose $t_0 = R^{\frac{3}{2}}$. Then we have $(-\tau)(t_0-\tau)\lesssim\abs{\tau}^2$, and
\begin{align*}
    \frac{M(-y,0-\tau)}{M(-y,t_0-\tau)}&\lesssim \exp\left\{-\frac{\abs{y}^2R^{\frac{3}{2}}}{\abs{\tau}^2}\right\}\\&\lesssim \exp\{-R^{\frac{1}{2}}\}\to 0 \,\,\text{as }R \to \infty.
\end{align*}

For $(y,\tau)\in F_R$, $\abs{\tau}\le R^{\frac{3}{2}}$, then $\abs{\tau}\le R\abs{y}^2$. Recall $t_0 = R^{\frac{3}{2}}$, we have
\[
\frac{M(-y,0-\tau)}{M(-y,t_0-\tau)}=\left(\frac{t_0-\tau}{-\tau} \right)^{\frac{n}{2}+1+s}\exp\left\{-\frac{\abs{y}^2}{4}\frac{t_0}{(-\tau)(t_0-\tau)}\right\}.
\]
Noticing that $(-\tau)(t_0-\tau)\lesssim |\tau|t_0$ and $\frac{t_0-\tau}{-\tau}\le 2\frac{t_0}{\abs{\tau}}$, we have
\begin{align*}
    \frac{M(-y,0-\tau)}{M(-y,t_0-\tau)} &\lesssim \frac{t_0}{\abs{\tau}}\exp\left\{-\frac{\abs{y}^2}{4\abs{\tau}}\right\} \\ &= \frac{t_0}{\abs{y}^2}\left(\frac{\abs{y}^2}{\abs{\tau}}\exp\left\{-\frac{\abs{y}^2}{4\abs{\tau}}\right\} \right)\\
    &\le \frac{R^{\frac{3}{2}}}{R^2}=R^{-\frac{1}{2}}\to 0\,\,\text{as }R\to \infty.
\end{align*}

In conclusion, we have proved that, for  any $(y,\tau)\in E_R\cup F_R$ and $|t|\leq \frac{R^2}{9},$
\[
\frac{M(-y,t-\tau)}{M(-y,t+t_0-\tau)}\to 0\,\,\text{as }R\to \infty,
\]
where $t_0=R^{\frac{3}{2}}\ll R^2$.
Hence,
\begin{align*}
F_i^6(0,t,R) &=\int _{E_R}u_i(y,\tau)M(0-y,t+t_0-\tau)\frac{M(-y,t-\tau)}{M(-y,t+t_0-\tau)}dyd\tau \\
& \leq o(1)\int _{\left(\R^n\times(-\infty,t)\right)\setminus Q_R}u_i(y,\tau)M\left(0-y,t+t_0-\tau\right)dyd\tau\\
&\leq  o(1)F_i(0,t+ t_0,R).
\end{align*}
For sufficiently large \( R \), since
\[
|t + t_0| \leq \frac{1}{9}R^2 + R^{\frac{3}{2}} \leq \frac{1}{4}R^2,
\]
it follows from \eqref{Fibdd} that
\[
F(0, t + t_0, R) \leq M+1,
\]
and thus we obtain
\[\limsup\limits_{i\to\infty}F_i^6(0,t,R)\to 0,\,\,\text{as }R\to \infty.\]
Similarly,
\[
\limsup\limits_{i\to\infty}F_i^7(0,t,R) \to 0,\,\,\text{as }R\to \infty.
\]
This verifies the second claim in \eqref{Ft}.

Since $F_i^6$ and $F_i^7$ are nonnegative, a similar argument in step 1 implies
\[
b(0,t) = \lim _{R\to\infty}\limsup _{i\to\infty} (F_i^4+F_i^5)(0,t,R).
\]
Consequently,  from \eqref{sFt} and \eqref{Ft}, we deduce
\begin{align*}
    b(0,t):=\lim _{R\to\infty}\limsup _{i\to\infty}F_i(0,t,R) =& \lim _{R\to\infty}\limsup _{i\to\infty} (F_i^4(0,t,R)+F_i^5(0,t,R))\\
     =&\lim _{R\to\infty}\limsup _{i\to\infty} (F_i^4(0,0,R)+ F_i^5(0,0,R))\\
     =& \lim _{R\to\infty}\lim _{i\to\infty} F_i(0,0,R)\\
    :=& b(0,0)\ge 0.
\end{align*}
In a conclusion, for any $(x,t)\in \R^n\times\R$, we arrive at $$b(x,t)=b(0,0).$$
Therefore, we verify \eqref{bx=b} with  $$b:=b(0,0)=c_{n,s}\lim _{R\to\infty}\lim_{i\to\infty}\int _{\left(\R^n\times(-\infty,0)\right)\setminus Q_R}\dfrac{u_i(y,\tau)}{|\tau|^{\frac{n}{2}+1+s}}e^{\frac{\abs{y}^2}{4|\tau|}}dyd\tau,
$$and complete the proof of Theorem \ref{main thm}.
\end{proof}
\section{Counterexamples}
In this section, we will present some examples to demonstrate that the constant $b$ in Theorem \ref{main thm} can indeed be positive. To begin with, we provide two special cases where the sequence of smooth functions  depends solely on the spatial variable \( x \) or solely on the temporal variable \( t \). Inspired by this, we construct another sequence of nonnegative  functions  that depends on both  variables \( x \) and  \( t \) to further consolidate Theorem  \ref{main thm2}.

 \begin{proof}[Proof of Theorem \ref{main thm2}]
Let's start with two simple examples.

\medskip

\leftline{\bf{Example 1. A sequence of functions depending only on the spatial variables}}
\medskip

Denote
\begin{equation}\label{ex1}\phi_j(x)=j^\alpha\phi(j^{-\beta}|x|), \alpha,\beta>0,\end{equation}
where $\phi(t)\in C_0^\infty([2,3],[0,1])$ is a nonnegative smooth function with compact support in $\R$.
Clearly,  $$\phi_j\to 0\ \mbox{in}\ C^2_{loc}(\R^n).$$
In addition, for any $x\in\R^n,$ let $j>>1$ such that $|x|\leq j^\beta$, then $\phi_j(x)=0,$ thus
\begin{align*}
     (-\Delta)^s\phi_j(x)=&\int_{\R^n}\frac{-\phi_j(y)}{|x-y|^{n+2s}}dy\\
    =&-j^{\alpha+\beta n}\int_{\R^n}\frac{\phi(|y|)}{|x-j^\beta y|^{n+2s}}dy\\
    =&-j^{\alpha-2\beta s}\int_{B_3\setminus B_2}\frac{\phi(|y|)}{|\frac{x}{j^\beta}- y|^{n+2s}}dy,
\end{align*}
which implies that
\begin{eqnarray}\label{ex1-C0}
    \lim\limits_{j\to\infty}(-\Delta)^s\phi_j(x)=\left\{\begin{array}{ll} 0,&\alpha<2\beta s,\\
    -\omega_{n-1}\int_{2}^3\frac{\phi(t)}{ t^{1+2s}}dt:=-C_0, &\alpha=2\beta s,
    \end{array}\right.
\end{eqnarray}
Therefore, in case of $\alpha=2\beta s,$
\begin{equation}\label{phij}
    \ma{s}\phi_j(x)=(-\Delta)^s\phi_j(x)\to-C_0\ \ \mb{for\ some}\ C_0>0,\ a.e.\ x\in\R^n.
\end{equation}

\medskip

\leftline{\bf{Example 2. A sequence of functions depending only on the time variable}}
\medskip

Similarly, another sequence of nonnegative smooth functions can be constructed as \begin{equation}\label{ex-t}\psi_j(t)=j^\alpha\psi(j^{-\beta}t), \alpha=\beta s>0,\end{equation}
with $\psi(t)\in C_0^\infty([-3,-2],[0,1])$, which satisfies
\[\psi_j\to 0\ \mbox{in}\ C^1_{loc}(\R),\]but
\begin{equation*}
    \ma{s}\psi_j(t)=\partial_t^s\psi_j(t)\to-C_1\ \ \mb{for\ some}\ C_1>0,\ a.e. t\in\R.
\end{equation*}
Here, the Marchaud fractional derivative $\partial_t^s$ is defined as in \eqref{1.00}.

\medskip

Obviously, the two examples above represent special cases: in \eqref{ex1}, \( \phi_j(x) \) is a sequence of smooth functions depending solely on the spatial variable \( x \), while in \eqref{ex-t}, \( \psi_j(t) \) depends only on the temporal variable \( t \). To further demonstrate that the constant \( b \) in Theorem \ref{main thm} can indeed be positive, we now construct a general sequence of nonnegative functions that depend on both spatial and temporal variables \( x \) and \( t \).

\medskip

\leftline{\bf{Example 3. A sequence with space-time coupled functions}}
\medskip

   Define  $\phi_j(x)$ as in \eqref{ex1} with $\alpha =2s, \beta=1 $, and let $C_0$ be the constant given  in \eqref{ex1-C0}. Denote \begin{equation}\label{def-eta}\eta_j(t):=\eta(j^{-\gamma}t)\ \mb{with}\  \eta(t):=(t_+)^2+1,  \gamma>s=\frac{\alpha}{2},\end{equation} and  \[w_j(x,t):=\frac{1}{C_0}\phi_j(x)\eta_j(t).\]

Clearly,  $w_j\in \mathcal{L}^{2s,s} (\R ^n\times\R)\cap C^{2,1} _{x,t}(Q_{j}) $ and $$w_j(x,t)\to 0,\ \mb{in}\ C^{2,1}_{x,t,loc}(\R^n\times\R).$$
  Next, we show that  \begin{equation}\label{ex2}\lim\limits_{j\to\infty}\ma{s}w_j(x,t)=-1,  a.e.\ (x,t)\in\R^n\times\R.\end{equation}

For each fixed $(x,t)\in\R^n\times\R,$ $\forall i>>1.$ It follows from the monotonicity of $\eta$ that \begin{align}\nonumber
    \ma{s}w_j(x,t)&=P.V.\int_{-\infty}^t\int_{\R^n} (w_j(x,t)-w_j(y,\tau))M(x-y,t-\tau)dyd\tau\\ \nonumber
    &=P.V.\frac{1}{C_0}\eta_j(t)\int_{-\infty}^t\int_{\R^n} (\phi_j(x)-\phi_j(y))M(x-y,t-\tau)dyd\tau\\ \nonumber
       &\ \ \ \ \ \ +P.V.\frac{1}{C_0}\int_{-\infty}^t\int_{\R^n} (\eta_j(t)-\eta_j(\tau))\phi_j(y)M(x-y,t-\tau)dyd\tau\\
       &\leq \frac{1}{C_0}\{\eta_j(t)(-\Delta)^s\phi_j(x)+\sup_{x\in\R^n}\phi_j(x)\partial_t^s\eta_j(t)\}. \label{sup-est}
\end{align}
  Note that \begin{equation}
  \label{D-eta}
      \partial_t^s\eta_j(t)=j^{-\gamma s}\partial_t^s \eta(j^{-\gamma}t)=C_sj^{-\gamma s}(j^{-\gamma}t_+)^{2-s}=C_sj^{-2\gamma }(t_+)^{2-s}.
      \end{equation}
Then by  \eqref{ex1},\eqref{phij}, \eqref{def-eta},  \eqref{sup-est}, and \eqref{D-eta}, we have
\begin{equation}\label{sup-est1}
    \ma{s}w_j(x,t)\leq \frac{(-\Delta)^s\phi_j(x)}{C_0}(j^{-2\gamma}t_+^2+1)+\frac{C_s}{C_0}j^{\alpha-2\gamma}t_+^{2-s}\to -1, \ \mb{as}\ j\to\infty.
\end{equation}
On the other hand, similar to the argument of \eqref{sup-est}, we also have
\begin{align}\nonumber
    \ma{s}w_j(x,t)&\geq \frac{1}{C_0}\{\eta_j(t)(-\Delta)^s\phi_j(x)+\inf_{x\in\R^n}\phi_j(x)\partial_t^s\eta_j(t)\}\\
    &\geq \frac{(-\Delta)^s\phi_j(x)}{C_0}(j^{-2\gamma}t_+^2+1) \to -1,\ \mb{as}\ j\to\infty.\label{sub-est}
\end{align}
Combining \eqref{sup-est1} and \eqref{sub-est}, we obtain
 \begin{equation*}
 \ma{s}w_j(x,t)\to -1, \ \mb{as}\ j\to\infty.
\end{equation*}
Thus , we  verify \eqref{ex2} and  complete the proof of Theorem \ref{main thm2}.

  \end{proof}

{\bf{Acknowledgements.}} Chen is
 partially supported by MPS Simons Foundation 847690. 
 
 Guo is partially supported by  National Natural Science
Foundation of China (NSFC Grant No. 12501145), the Postdoctoral Fellowship Program of CPSF (No.GZC20252004),
the Natural Science Foundation of Shanghai (No.25ZR1402207), and the China Postdoctoral
Science Foundation (No.2025T180838 and 2025M773061).
 
 Li and Ouyang is partially supported by National Natural Science
Foundation of China (NSFC Grant No. W2531006, 12031012, and 11831003) and the Institute of Modern Analysis-A Frontier Research Center of Shanghai.
 \medskip

{\bf{Date availability statement:}} Data will be made available on reasonable request.
\medskip

{\bf{Conflict of interest statement:}} There is no conflict of interest.

\bigskip

Wenxiong Chen

Department of Mathematical Sciences

Yeshiva University

New York, NY, 10033, USA

wchen@yu.edu
\medskip

Yahong Guo

School of Mathematical Sciences

Shanghai Jiao Tong University

Shanghai, 200240, P.R. China

yhguo@sjtu.edu.cn
\medskip

Congming Li

School of Mathematical Sciences

Shanghai Jiao Tong University

Shanghai, 200240, P.R. China

congming.li@sjtu.edu.cn
\medskip

Yugao Ouyang

School of Mathematical Sciences

Shanghai Jiao Tong University

Shanghai, 200240, P.R. China

ouyang1929@sjtu.edu.cn

\end{document}